\documentclass[a4paper,12pt]{article}
\usepackage[utf8]{inputenc}
\usepackage[american]{babel}
\usepackage{anysize}

\usepackage{amsmath}
\usepackage{amssymb}
\usepackage{latexsym}
\usepackage{mathrsfs}

\usepackage{amsthm}

\usepackage{indentfirst}

\usepackage{color}

\usepackage{ifpdf}

\ifpdf
\usepackage[pdftex]{graphicx}
\else
\usepackage{graphicx}
\fi

\newtheorem{theorem}{Theorem}	
\newtheorem{lemma}[theorem]{Lemma}

\theoremstyle{definition}

\theoremstyle{remark}

\newcommand{\bs}{\mathop{\rm BS}\nolimits}
\newcommand{\conv}{\mathop{\rm conv}\nolimits}

\begin{document}
\sloppy
\ifpdf
\DeclareGraphicsExtensions{.pdf, .jpg, .tif, .mps}
\else
\DeclareGraphicsExtensions{.eps, .jpg, .mps}
\fi

\title{Balanced 2--subsets}

\author{Mikhail V. Bludov and Oleg R. Musin}

\date{}

\maketitle

\begin{abstract} 
Balanced sets appeared in the 1960s in cooperative game theory as a part of nonempty core conditions. In this paper we present  a classification of balanced families containing only 2--element subsets. We also discuss generalizations of the classical Sperner and Tucker lemmas using balanced sets.  
\end{abstract}

\section{Introduction} 

Balanced sets as families of weighted subsets of a finite set first appeared in Bondareva \cite{Bon} and Shapley \cite{Sh67} papers. 

Denote by $[d]$ the set $\{1,2,...,d\}$. Let $\Phi$ be a family of subsets $\{S_1,...,S_m\}$ of $[d]$. Following Shapley \cite{Sh}, this family is called {\em balanced} if there is a set of non--negative weights $\{w_1,...,w_m\}$ such that $$
\sum\limits_{k=1}^m{w_k\eta_k=(1,...,1)},
$$ where $\eta_k$ is the characteristic (indicator) vector of $S_k$ in $[d]$. A balanced family $\Phi$ is called {\em minimal} if there are no proper balanced subfamilies in $\Phi$.

\medskip

Now we consider families that contain only 2--element subsets. We say that a family $F$ is of {\em odd size} if it contains an odd number of subsets. 

Let $I=\{i_1,...,i_n\}$, where $n\ge3$.  We say that a family $F$ of subsets from $I$ is {\em cyclic with respect to $I$} if 
$F=\{(i_1,i_2),(i_2,i_3),...,(i_n,i_1)\}$. If $n=2$, then there is only one $2$--subset of $I$. In this case we call $F=\{(i_1,i_2)\}$ {\em isolated}. 

The proof of the following theorem is given in Section 2. 
\begin{theorem}
Let $\Phi$ be a minimal balanced family of $2$--subsets in $[d]$.  Then $[d]$ is the disjoint union of subsets $I_1,...,I_k$ and $\Phi=\{\Phi_\ell\}_ {\ell=1,...,k}$, where $\Phi_\ell$ is either cyclic of odd size with respect to $I_\ell$ or it is isolated. 
\end{theorem}

These decompositions correspond to partitions of $d$ into parts 2, 3, 5,... 
The generating function of this partition is 
$$
\frac{1}{1+x}\prod\limits_{i=0}^\infty{\frac{1}{1-x^{2i+1}}}.
$$
It is easy to show the connection with the odd partitions. Denote by $b(d)$ the number of our partitions. Let $q(d)$ be the number of odd partitions. Then 
$$
b(d)=q(d)-q(d-1)+...+(-1)^dq(0).
$$ 

We can apply theorem 5 for cooperative games with 2-players coalitions. From the Bondareva -- Shapley theorem it follows that we can check the non--empty core condition only on the minimal balanced sets. Since we know a classification of these sets we can simplify the condition for the non-emptiness of the core.  

\medskip


Here we give a general geometrical definition of balanced sets.

Let  $V=\{v_1, v_2, \dots, v_m\}$ be a set of points in $\mathbb{R}^{d}$.  
A subset $\{v_{i_1}, \dots, v_{i_k}\}$ is called {\em balanced} if $c_V$ lies in the convex hull $\text{conv}(v_{i_1}, \dots, v_{i_k})$, where $c_V$ is the center of mass of $V$. The corresponding set of indices $I=\{i_1, \dots, i_k\}$ is also called balanced.

A balanced subset of $V$ is {\em minimal} if and only if it does not contain a subset that is balanced. Denote the family  of {minimal balanced subsets} as $\bs(V)$.

Sperner's lemma on colorings of triangulations vertices and its extension to coverings  Knaster –Kuratowski -- Mazurkiewicz (KKM) lemma are discrete analogs of the Brower's fixed point theorem. KKM may be extended to KKMS theorem, see \cite{Sh,SV}. All these theorems have many applications, particularly in game theory and mathematical economics.

Tucker, Ky Fan, and Shashkin's lemmas are discrete versions of the Borsuk--Ulam theorem. They also have various applications. 

In Section 3, we consider Theorems A and B, which are generalizations of discrete versions of fixed point theorems that rely on balanced sets of $V$. If for some $V$ we know its $\bs(V)$, then that gives explicit versions of these theorems.


\section{Proof of Theorem 1}

Let $\Phi$ be a family that contains only 2--element subsets from $[d]$. This family can be described geometrically with an orthonormal basis $e_1,\dots, e_{d}$ of $\mathbb{R}^{d}$. To every $2$-element subset $(i,j)$ we assign a point in $\mathbb{R}^{d}$: 
\[
e_{ij}:=\frac{1}{2}(e_i+e_j)\; \; (1\leqslant i<j\leqslant d).
\]
This set of points we denote by $V_d$. Obviously, all these ${d}\choose {2}$ points in $V_d$ lie in a hyperplane $\Pi_d$ that defined by an equation $x_1+\ldots+x_{d}=1$. 

(Note that the points $V_d$ are midpoints of edges of a $(d-1)$-dimensional simplex in $\Pi_d$ with the vertex set $e_1,...,e_d$.
A polytope $P_d:=\conv(V_d)$ plays an important role in discrete geometry, graph theory and coding theory. In particular, $V_d$ is an example of a $2$--distance set in $\mathbb{R}^{d-1}$. There are many cases when the number of points in a maximal 2--distance set is at most $|V_d|=$${d}\choose {2}$.)

It is easy to see that there is one--to--one correspondence between balanced subsets of $V_d$ and balanced families  $\Phi=\{S_1,...,S_m\}$, where all $S_i$ are 2-subsets of $[d]$.

Denote by $c_d$ the center of mass of $V_d$. Then $c_d=(\frac{1}{d}, \dots, \frac{1}{d})\in \Pi_d$. The main goal of this section is to describe the set of all minimal balanced sets $S \subset V_d$. 

\medskip

Let $K_d$ be a complete graph on $d$ vertices  $\{a_1,...,a_d\}$. We will identify the vertices $\{a_1,...,a_d\}$ with the vectors $\{e_1,...,e_d\}$ and with their indices $\{1,2,...,d\}$. 

Let $S$ be a subset of $V_d$. Define a graph $G(S)$ as a subgraph of $K_d$ by the rule: {\em $(i,j)$ is an edge of $G(S)$ iff $e_{ij} \in S$.} 

Denote by $n(S)$ the number of vertices of $G(S)$. Note that the edges of this graph correspond to the elements of $S$. It is clear that  
$$
S\subseteq V_{n(S)}\subset \Pi_{n(S)}\subset \mathbb R^{n(S)} \subseteq  \mathbb R^d
$$ 
and we have the following statement.

\begin{lemma} Let $S$ be a subset of $V_d$. Suppose that  $S=S_1\cup S_2$, where $S_1\cap S_2=\emptyset$ and $G(S)$ is the disjoint union of graphs $G(S_1)$ and $G(S_2)$. Then
$$
 S_i\subset \mathbb R^{n(S_i)}, \: i=1,2; \quad  \mathbb R^{n(S)}= \mathbb R^{n(S_1)} \oplus \mathbb R^{n(S_2)}
$$
\end{lemma}

This lemma yields that if $G(S_1),...,G(S_k)$, $k>1$,  are connected components of $G(S)$ then the sets $S_i$ lie in mutually orthogonal subspaces of $\mathbb R^{n(S)}$. 

\begin{lemma} 
Let $S\subset V_d$ be a balanced set. Suppose  $G(S_1),...,G(S_k)$ are connected components of  $G(S)$. Then for all $i=1, \dots, k$ the set $S_i$ is also balanced for $V=V_{n(S_i)}$.  
\end{lemma}

The next lemma follows directly from Shapley's  definition:

\begin{lemma} Let $S$ be a minimal balanced set, $|S|>1$. Suppose that the graph $G(S)$ is connected. Then $G(S)$ has no vertices of degree 1. 
\end{lemma}

In Section 1 we defined cyclic and isolated families. For $S\subset V_n$ these definitions mean: $S$ is cyclic if  $G(S)$ is a polygon with $n$--vertices and $S$ is isolated if $n=2$ and so $S$ contains only one vertex. 

\begin{lemma} Let $S$ be a minimal balanced set of $V=V_n$ with $n>2$. Suppose that  $G(S)$ is connected and $n(S)=n$. Then  $n$ is odd and $S$ is cyclic.
\end{lemma}

\begin{proof} Our proof relies on the following well--known theorem:

\medskip

\noindent {\bf Carath\'eodory's theorem.} {\em If a point $x$ lies in the convex hull of a set $P$ in $\mathbb R^m$, then $x$ can be written as the convex combination of at most $m+1$ points in $P$.}

\medskip

Note that $S$ is a subset of $(n-1)$--dimensional Euclidean space. By the assumption the number of vertices of $G(S)$ is  $n$ and  $c_n$ lies in $\conv(S)$.

Caratheodory's theorem implies that $c_n$ is the convex combination of $\ell\le n$ vertices from $S$. From the minimality of $S$ it follows that this subset of vertices coincides with $S$ and $|S|=\ell$, hence $|S|\le n$.

We see that the graph $G(S)$ is on $n$ vertices and has at most $n$ edges. Then from Lemma 4 it follows that this graph is an $n$--polygon.

Suppose $n=2k$, $S=\{(1,2)(2,3),...,(2k,1)\}$, and $S'=\{(1,2)(3,4),...,(2k-1,2k)\}$. Then $|S'|=k$. Since $c_n$ is covered by the convex hull of $S'$ and $S'\subset S$, we see that $S$ is not minimal and $n$ can not be even.
\end{proof}

{\bf Theorem 1 directly follows  from lemmas 3 and 5.}


\section{Tucker, Fan, and Shashkin lemmas as corollaries of the balanced sets theorem}

\subsection{Discrete versions of fixed point theorems.} 
Let a space $X$ is covered by $m$ open (or closed) sets and $V=\{v_1, v_2, \dots, v_m\}\subset\mathbb{R}^{d}$. Following  \cite{MusH}, we can construct a map $f_V:X\to \text{conv}(V)$. 

Let $T$ be a triangulation of a manifold $M$. A vertex coloring $L:V(T)\to \{1,...,m\}$  is a special case of covering. Then we can define a map $f_L:T\to \text{conv}(V)$. Main results about these maps may be found in  \cite[Theorem 3.1]{MusH},  \cite[Cor. 3.2]{MusH}, and \cite[Th. 4.2]{MusBC}. Here we give corollaries of these theorems.

\medskip

\noindent {\bf Theorem A.} {\em Let $V:=\{v_1, \dots, v_m\}\subset \mathbb{R}^{d}$. Let $F=\{F_1, \dots, F_m\}$ be a closed (or open) covering of $n$-dimensional disc. Suppose that $F$ is such that  $f_V$ is not null--homotopic on the boundary. Then there is a minimal  balanced set $I\in\bs(V)$ such that $\cap_{i \in I} F_i \neq \emptyset$. }

\medskip 

\noindent {\bf Theorem B.}    {\em Let $V:=\{v_1, \dots, v_m\}\subset \mathbb{R}^{d}$. Let $T$ be a triangulation of $n$-dimensional disc and $L:V(T)\to \{1,...,m\}$ be a coloring of $V(T)$.  Suppose that $f_L$ is not null--homotopic on the boundary. Then there is a simplex $s$ in $T$ and  $I\in\bs(V)$ such that vertices of $s$ are colored with all colors from $I$. }

\medskip 

Note that the {\em ``not null-homotopic''} on the boundary condition is true in the case of classical fixed point theorems. In particular, {\em Sperner's coloring},  {\em KKM's covering}, and  {\em antipodal coloring} on the boundary are special cases of this condition.

Using theorems A and B  we can obtain discrete versions of fixed point theorems.

\medskip 

Suppose $m=d+1$, $n=d$. Let $F$ be a covering of $d$--dimensional simplex $\Delta^d \subset \mathbb{R}^{d}$ with the vertex set $V$. Assume that $F$ satisfies the boundary conditions of the KKM lemma. Then this covering is not null--homotopic on the boundary of $\Delta^d$ and the KKM lemma follows from Theorem A. Sperner's lemma can be easily deduced from the KKM lemma or from Theorem B.

\medskip 

\noindent {\em Every Sperner coloring of a triangulation of $\Delta^d$ contains a cell whose vertices all have different colors.}

\medskip 

 Shapley's KKMS lemma \cite{Sh,SV} also can be easily deduced from Theorem A, see \cite[Cor. 4.2]{MusBC}.  In this case the set of points $V$ is the set of all centers of mass of $k$-vertex subsets of $\Delta^d$, where $1\le k \le d$.  
 
 \medskip

Let  $V=V_d$ and we are in the conditions of Theorems A and B. {\em Then Theorem 1 yields a new result for colorings with ${d} \choose {2}$ colors.}

 \subsection{Antipodal balanced 2-subsets.}

  Let $P$ be a centrally symmetric polytope in $\mathbb{R}^{d}$ with the vertex set $V=V(P)$. In other words, if $v\in V$, then $(-v)\in V$. We see that the center of mass  $c_V=O$, where $O$ is the origin of $\mathbb{R}^{d}$. It is clear that the {\em family of balanced  $2$-subsets of $V$ is the set of all antipodal pairs $(v,-v)$, where $v\in V$.}

Suppose that $T$ and $L$ from Theorem B are both antipodally symmetric on the boundary. Then (see \cite{MusS,MusH,MusBC}) $f_L$ is not null-homotopic on the boundary, hence we can use Theorem B.
  
 \medskip 

 \noindent {\bf 1.}   Let  $e_1,\dots, e_{d}$ be a standard orthonormal basis for $\mathbb{R}^{d}$. Let $P$ be a regular cross polytope with the vertex set $V=\{\pm e_1,...,\pm e_d\}$. Then we see that $\bs(V)$ coincides with the set of all pairs of antipodal vertices. Then Tucker's lemma follows from Theorem B.

\medskip 

\noindent  {\em Let $T$ be a triangulation of a $d$-dimensional disc  such that $T$ is antipodally symmetric on the boundary. Let $L : V(T) \rightarrow  \{+1, -1, +2, -2, \dots , +d, -d\}$ be a coloring that is antipodal (i.e. $L(-v) = -L(v)$) for  all vertices $v$ on the boundary. Then there exists a complementary edge, i.e.  $[u,v]\in T$ such that  $L(u)=-L(v)$.}

\medskip 

\noindent {\bf 2.} In \cite[Th. 5.2]{MusSpT} was constructed a convex polytope $P(n,d) \subset \mathbb{R}^{d}$ with $2n$ centrally symmetric vertices $V=\{\pm v_1, \dots,\pm  v_n\}$ such that $\bs(V)$ consists of antipodal pairs $(v_i,-v_i)$, $i=1,...,n$ and $d$-simplices with vertices $(v_{k_0}, -v_{k_1}, \dots,(-1)^{d}v_{k_d})$ or $\{-v_{k_0}, v_{k_1}, \dots, (-1)^{d+1}v_{k_d})$, where $1 \leq k_0 < \dots < k_d \leq n$. Then Theorem B yields Ky Fan's lemma: 

\medskip 

\noindent {\em Let $T$ be a triangulation of a $d$-dimensional disc that is antipodally symmetric on the boundary. Let $L : V(T) \rightarrow  \{+1, -1, +2, -2, \dots , +n, -n\}$ be a coloring  that is antipodal on the boundary. Suppose that there are no complementary edges in $T$. Then there are an odd number of alternating $d$-simplices, i.e. simplices that are colored by $(k_0, -k_1, k_2, \dots ,(-1)^dk_d)$, where $1 \leq |k_0| < \dots < |k_d| \leq n$ and all $k_i$ are of the same sign.}
  
\medskip 

\noindent {\bf 3.} In  \cite{MusS} is considered an extended version of Shashkin's lemma.
  
\medskip 

\noindent {\em Let $T$ be a triangulation of a $(d-1)$-dimensional disc that is antipodally symmetric on the boundary. Let $L : V(T) \rightarrow  \{+1, -1, +2, -2, \dots , +d, -d\}$ be a coloring that is antipodal on the boundary. Suppose that there are no complementary edges in T. Then for every set of colors $\Lambda=\{\ell_1, \dots,\ell_d\}$, where $|\ell_i|=i$ for $i=1, \dots, d$, there are an odd number of cells in $T$ that are labelled by $\Lambda$ or $(-\Lambda)$.}

\medskip 

The proofs in \cite{MusS} do not rely on Theorem B. Here we show that Shashkin's lemma follows from this theorem.

Let $\Delta$ be a $(d-1)$--dimensional simplex in $\mathbb{R}^{d-1}$ with the center of mass at the origin $O$. Let $v_1,...,v_d$ be the vertices of $\Delta$. Suppose that $V$ is the set of points $\{\pm v_1,...,\pm v_d\}$. It is easy to prove that $\bs(V)$ consists of the set of all pairs $(v_i,-v_i)$ and the sets $\{v_1,...,v_d\}$ and $\{-v_1,...,-v_d\}$. Assign a color $\ell_i$ to a vertex $v_i$ and $(-\ell_i)$ to $-v_i$. Then $f_L$ is antipodal on the boundary and Shashkin's lemma follows from Theorem B.

 \medskip
 
 M. V. Bludov, MIPT
 
  {\it E-mail address:} michaelbludov@gmail.com
  
  \medskip 

 O. R. Musin,  University of Texas Rio Grande Valley, School of Mathematical and
 Statistical Sciences, One West University Boulevard, Brownsville, TX, 78520, USA.

 {\it E-mail address:} oleg.musin@utrgv.edu


\begin{thebibliography}{10}


\bibitem{Bon} 
O. N. Bondareva. Some applications of linear programming methods to the theory of cooperative games. {\em Problemy Kibernetiki,} {\bf 10}  (1963),  119–139.

\bibitem{MusSpT} 
O. R. Musin, Extensions of Sperner and Tucker's lemma for manifolds, {\it J. of Combin. Theory Ser. A,} {\bf 132} (2015), 172--187.  

\bibitem{MusS}
O. R. Musin, Generalizations of Tucker--Fan--Shashkin lemmas,  {\it Arnold Math. J.}, {\bf 2:3} (2016),  299--308. 

\bibitem{MusH} 
O. R. Musin, Homotopy invariants of covers and KKM type lemmas, {\it Algebr. Geom. Topol.,} {\bf 16} (2016), 1799--1812. 

\bibitem{MusBC}
O. R. Musin, KKM type theorems with boundary conditions, {\it J. Fixed Point Theory Appl.,} {\bf 19} (2017), 2037-2049.

\bibitem{Sh67}
  L. S. Shapley. On balanced sets and cores, {\em Naval Res. Logist. Quart.,} {\bf 14} 
(1967), 453--460.

\bibitem{Sh}
L. S. Shapley  On balanced games without side payments, in
{\it Mathematical Programming,} Hu, T.C. and S.M. Robinson (eds), Academic
Press, New York, 261--290, 1973.

\bibitem{SV}
L. S. Shapley and R. Vohra, On Kakutani's fixed point theorem, the KKMS theorem and the core of a balanced game, {\it Economic Theory},  {\bf 1} (1991), 108--116.

\end{thebibliography}
\end{document}